\documentclass{arxiv} 

\newif\ifarxiv
\arxivtrue

\usepackage{filecontents}

\usepgfplotslibrary{statistics, patchplots}
\usetikzlibrary{patterns}

\pgfplotsset{
    boxplotstyle/.style = {
        figurestyle,
        boxplot/draw direction = x,
        yticklabels = {0.05, 0.15, 0.25, 0.35, 0.45},
        ytick = {1,2,3,4,5},
        cycle list = {{fill = blue, draw = black, thick}, {pattern = crosshatch, pattern color = red, draw = black, thick}},
        boxplot/box extend = 0.25
    },
    colorbarstyle/.style = {
        axis line style = thick,
        at = {(1.1, 0)},
        anchor = south west,
        height = 0.35\linewidth,
        title style = {font = \small}
    },
    every axis legend/.append style = {
        at = {(0.5,1.02)},
        anchor = south,
        legend columns = 2,
        font = \small,
        draw = none
    }
}
\newcommand{\y}[1][]{\vc{y}_{#1}}
\newcommand{\yhat}[1][]{\vchat{y}_{#1}}
\newcommand{\ytrue}[1][]{\vc{x}_{#1}}
\newcommand{\pole}[1][]{p_{#1}}
\newcommand{\estpole}[1][]{\widehat{p}_{#1}}

\newcommand{\misfit}{\vctilde{y}}

\newcommand{\toep}[2]{\mt{T}_{#1}^{#2}}
\newcommand{\tnum}[2][1.4]{\tablenum[table-format = #1]{#2}}

\title{Incorporating Fixed-Pole Information in the Data-Driven Least Squares Realization Problem\thanks{This work was supported in part by the KU Leuven Research Fund (grants IBOF/23/064, C3/20/117, and C3I/21/00316); in part by the FWO (grants S005319 and T001919N); in part by the Departement Economie, Wetenschap \& Innovatie via the Flanders AI Research Program; in part by the Vlaams Agentschap Innoveren \& Ondernemen (grant HBC/2021/0076); and in part by the European Research Council (grant 885682). The work of Sibren Lagauw was supported by an FWO fellowship (grant FR/11K5625N).}}

\author[$\dagger$, $\ddagger$, $\S$]{Christof Vermeersch}
\author[$\ddagger$, $\S$]{Sibren Lagauw}
\author[$\ddagger$]{Bart De Moor}

\affil[$\dagger$]{Corresponding author (\url{christof.vermeersch@esat.kuleuven.be})}
\affil[$\ddagger$]{Center for Dynamical Systems, Signal Processing, and Data Analytics (STADIUS), Dept. of Electrical Engineering (ESAT), KU Leuven, Kasteelpark Arenberg 10, 3001 Leuven, Belgium}
\affil[$\S$]{Both authors contributed equally to this work}

\begin{document}
    \maketitle
    
    \begin{abstract}
    In practical least squares realization problems, partial information about the pole locations of the dynamical model may be known \textit{a priori}. 
    Existing techniques for incorporating this prior knowledge, such as prefiltering the given data, are typically heuristic and lack theoretical guarantees. 
    We extend our previously developed globally optimal estimation approach to accommodate fixed poles in the least squares realization problem.
    In particular, we reformulate the problem as a (rectangular) multiparameter eigenvalue problem, the eigenvalues of which characterize all local and global minimizers of the constrained estimation problem. 
    We present numerical examples to demonstrate the effectiveness of the proposed method and experimentally validate the paper's central hypothesis: incorporating \textit{a priori} information on the poles enhances the estimation results.
\end{abstract}

\section{Introduction}
    \label{sec:introduction}

    Given a sequence of (discrete-time) output data, the \emph{standard least squares realization problem} addresses the question ``how can we modify the given data points in a least squares sense so that they become the output of an autonomous, single-output, linear time-invariant dynamical model, and what are the corresponding model parameters?''
    For a predetermined model order, the realized model is completely characterized by the coefficients or roots of its characteristic equation. 
    Although the model class is linear in the parameters, the least squares realization problem of finding the smallest adaptation of the observed output data is nonlinear in both the model parameters and misfit, resulting in a non-convex optimization problem with potentially many local optima (see \cref{fig:contour}).

    Many approaches for determining the coefficients or poles of the model exist, with first occurrences in the literature dating back to Prony in the 18th century~\cite{Hauer1990}.
    Recently, globally optimal approaches that minimize the misfit between the observed and modified output data have been proposed by some of the authors~\cite{DeMoor2020, Lagauw2025}.  

    Most methods for solving the least squares realization problem, including the globally optimal approaches in~\cite{DeMoor2020} and~\cite{Lagauw2025}, only consider the observed output data and a predetermined model order.
    However, some of the model's poles may be known (``fixed'') in practical settings. 
    Such situations commonly arise in process control, structural dynamics, and econometrics, where certain modes are either predetermined or estimated from first principles~\cite{Auton1981}.
    Similarly, when estimating parameters for exponential signals, information about the signal properties is often available from the physical setup, previous experiments, or dynamics introduced to model the noise disturbances~\cite{Chen1997,Chen1997a}.
    Using this information improves the estimation quality in a certain sense---\emph{the paper's central hypothesis.}
    Therefore, we extend our globally optimal method developed in~\cite{Lagauw2025} to incorporate fixed-pole information, specifically improving the retrieval of the exact system output dynamics in a noisy context.

    The idea that using information about the fixed-pole locations may improve the estimation quality is not new. 
    Three alternative versions of Prony's method were already described in~\cite{Auton1981}, each constraining the original method in a different way to handle the fixed poles.
    The linearly constrained total least squares (TLS) method employs linear constraints to incorporate \textit{a priori} knowledge about the poles of the model~\cite{Dowling1994}.  
    And, in the context of nuclear magnetic resonance spectroscopy, the Hankel total least squares (HTLS) method has been extensively used for spectral estimation~\cite{Hua1991}, with several variants, the so-called HTLS-PK methods~\cite{Chen1997a, Chen1997}, incorporating \textit{a priori} knowledge on the signal properties.
    Alternatively, rather than adapting parameter estimation techniques to deal with the fixed poles, one could also prefilter the observed output data in an attempt to remove the effects of the fixed poles before estimation. 
    A naive approach is to start the realization procedure from the optimal misfit associated to the lower-order model defined by the fixed poles.
    The time series deflation (TSD) technique~\cite{Majda1989} is a second heuristic that applies a moving-average prefilter to the observed output data to remove the known components. 

    Although these heuristic methods provide practical solutions, a globally optimal approach for dealing with fixed poles does not yet exist.
    The approach that we take in this paper starts from our globally optimal methodology developed in~\cite{Lagauw2025}, the so-called \emph{standard globally optimal realization (S-GOR) technique}.
    If one or more poles of the model are known beforehand, then the new \emph{fixed-pole globally optimal realization (FP-GOR) technique} finds the globally optimal approximation of the remaining, unknown poles. 
    We end up with a (rectangular) multiparameter eigenvalue problem (MEP), which is fundamentally different than the above-mentioned heuristics.
    It is the first globally optimal approach that takes the fixed poles into account, replacing engineering practice by exactness at the cost of a higher computational complexity.
    The computational complexity, however, is lower than of the standard technique since the number of poles to estimate is lower.
    We validate the paper's central hypothesis that \textit{a priori} information improves the estimation under noise in the last numerical example; the model-compliant output data estimated by our new FP-GOR approach are, in a noisy context, closer to the underlying exact system outputs than those obtained via S-GOR.

    The remainder of the paper continues as follows:
    We give a mathematical description of the problem and a motivational example in \cref{sec:problemdefinition}. 
    Next, in \cref{sec:methodology}, we explain how we obtain, via the first-order optimality conditions, an MEP. 
    \emph{The MEP reformulation of the fixed-poles least squares realization problem is our main contribution.}
    We solve three examples numerically in \cref{sec:numericalexamples}, including an extensive experiment with different noise levels.
    Finally, we conclude this paper and look towards future work in \cref{sec:conclusion}.

    \begin{figure}
        \centering
        \input{figures/contour.tikz}
        \vspace{-0.2cm}
        \caption{Misfit $\norm{\misfit}_2^2$ of the least squares realization problem in \cref{ex:motivational} plotted against the model parameters $a_1$ and $a_2$, normalized by setting $a_0 = 1$. The surface shows the non-convex nature of the underlying optimization problem, with the critical points of the standard least squares objective function being minimizers (~\ref{plot:std-minimum}~), saddle points (~\ref{plot:std-saddlepoint}~), and maximizers (~\ref{plot:std-maximum}~). Given a fixed pole, the fixed-pole least squares realization problem is subject to an additional constraint (\ref{plot:constraint}). A different minimizer (~\ref{plot:fp-minimum}~) and maximizer (~\ref{plot:fp-maximum}~) are obtained in this example. The two heuristic solutions discussed in \cref{ex:motivational} are also shown (~\ref{plot:heuristics}~).}
        \label{fig:contour}
        \vspace{-0.4cm}
    \end{figure}

\section{Problem definition and motivational example}
    \label{sec:problemdefinition}

    We consider an autonomous, single-output, linear time-invariant dynamical model in discrete time with a predefined model order $n$ that corresponds to the number of states.
    Model-compliant output data $\yhat \in \Rset^N$, with $N > 2 n$ data points $\widehat{y}_k$, satisfy a difference equation of the form 
    \begin{equation}
        \label{eq:differenceequation}
        a_0 \widehat{y}_{k + n} + a_1 \widehat{y}_{k + n - 1} + \ldots + a_n \widehat{y}_k = a(z) \widehat{y}_k = 0,
    \end{equation}
    for all $k = 1, 2, \ldots, N - n$, where $a(z) = a_0 z^n + a_1 z^{n - 1} + \ldots + a_n$ is a degree $n$ polynomial in the forward-shift operator $z$ (i.e., $z \widehat{y}_k = \widehat{y}_{k + 1}$).
    Writing out~\eqref{eq:differenceequation} for every $k$, 
    \begin{equation}
        \label{eq:toeplitzdefinition}
        \begin{bmatrix}
            a_n  & \cdots & a_1 & a_0 & 0 & \cdots & 0 \\
            0 & a_n  & \cdots & a_1 & a_0  & \ddots & \vdots \\
            \vdots  & \ddots & \ddots  &  & \ddots  & \ddots & 0 \\
            0  & \cdots & 0 & a_n & \cdots & a_1 & a_0 \\
        \end{bmatrix} \yhat = \vc{0}
    \end{equation} 
    expresses the model via the (banded) Toeplitz matrix $\toep{N - n}{\vc{a}} \in \Rset^{(N - n) \times N}$ for the coefficient vector $\vc{a} = \begin{bmatrix} a_n & \cdots & a_1 & a_0 \end{bmatrix}^\tr \in \Rset^{n + 1}$. 
    The kernel of $\toep{N - n}{\vc{a}}$ is spanned by Vandermonde vectors generated by the roots of $a(z)$.
    Consequently, the model-compliant output data $\yhat$ can be expressed as a weighted sum (with coefficients $x_i$) of exactly $n$ linearly independent (confluent\footnote{Since our methodology uses the coefficients of $a(z)$ to parameterize the model, it also holds for models with multiple poles (see how in~\cite{DeMoor2020}).}) Vandermonde vectors $\vc{v}_i = \begin{bmatrix} 1 & \pole[i] & \cdots & \pole[i]^{N - 1} \end{bmatrix}^\tr \in \Cset^N$ for each pole $\pole[i]$:
    \begin{equation}
        \label{eq:image}
        \yhat = \sum_{i = 1}^n \vc{v}_i x_i.
    \end{equation}

    \begin{definition}
        Given observed output data $\y \in \Rset^N$, the \emph{standard least squares realization problem} seeks to determine the model parameters $\vc{a} \in \Rset^{n + 1}$ of the $n$th-order model that minimizes the $l_2$ norm of the misfit $\misfit = \y - \yhat$. 
        It corresponds to the following optimization problem:
        \begin{equation}
            \label{eq:fulloptimizationproblem}
            \begin{gathered}
                \min_{\vc{a}, \yhat} \frac{1}{2} \norm{\misfit}^2_2 =\min_{\vc{a}, \yhat} \frac{1}{2} \norm{\y - \yhat}^2_2, \\
                \sub \toep{N - n}{\vc{a}} \yhat = \vc{0}, \left\langle \vc{a} , \vc{e} \right\rangle = 1,\\ 
            \end{gathered}
        \end{equation} 
        where the last constraint avoids the trivial solution $\vc{a} = \vc{0}$. 
    \end{definition}

    The \emph{normalization vector} $\vc{e} \in \Rset^{n + 1}$ in~\eqref{eq:fulloptimizationproblem} must have at least one non-zero element.
    A commonly used choice is $\vc{e} = \begin{bmatrix} 0 & \cdots & 0 & 1 \end{bmatrix}^\tr$, which boils down to setting $a_0 = 1$. 

    Suppose that $m < n$ poles, $\{ \pole[i] \}_{i = 1}^m$, of the model are fixed, such that the model can be factorized as 
    \begin{equation}
        \label{eq:factorization}
        a(z) = b(z) \prod_{i = 1}^m (z - \pole[i]) = b(z) c(z), 
    \end{equation} 
    where $b(z)$ is a $q$th-order polynomial in the forward-shift $z$ with unknown coefficients $\vc{b} = \begin{bmatrix} b_q & \cdots & b_1 & b_0 \end{bmatrix}^\tr \in \Rset^{q + 1}$, in which we introduce $q = n - m$ for notational convenience.
    The coefficients $\vc{c} = \begin{bmatrix} c_m & \cdots & c_1 & 1 \end{bmatrix}^\tr \in \Rset^{m + 1}$ of the $m$th-order polynomial $c(z) = \prod_{i = 1}^m (z - \pole[i])$ follow directly from the fixed poles.
    After defining (banded) Toeplitz matrices $\toep{N - q}{\vc{b}}$ and $\toep{N - n}{\vc{c}}$ from $\vc{b}$ and $\vc{c}$, the factorization in~\eqref{eq:factorization} can be used to rewrite~\eqref{eq:toeplitzdefinition} as
    \begin{equation}
        \label{eq:filter}
        \toep{N - n}{\vc{a}} \yhat = \toep{N - n}{\vc{c}} \toep{N - q}{\vc{b}} \yhat = \vc{0}. 
    \end{equation} 
    The order of $\toep{N - q}{\vc{b}}$ and $\toep{N - n}{\vc{c}}$ can be changed as long as the dimensions are adjusted appropriately. 
    Indeed,
    \begin{equation}
        \label{eq:swapped}
        \toep{N - n}{\vc{c} \cdot \vc{b}} = \, \toep{N - n}{\vc{c}} \toep{N - q}{\vc{b}} = \toep{N - n}{\vc{b}} \toep{N - m}{\vc{c}}. 
    \end{equation}
    Requiring the model $a(z)$ to have certain poles in~\eqref{eq:image} fixes $m$ of the Vandermonde vectors:
    \begin{equation}
        \label{eq:restrictedimage}
        \yhat = \sum_{i = 1}^m \vc{v}_i x_i + \sum_{i = m + 1}^n \vc{v}_i x_i = \sum_{i = 1}^m \vc{v}_i x_i + \yhat'.
    \end{equation}

    \begin{definition}
        Given observed output data $\y \in \Rset^N$ and $m$ fixed poles $\{ \pole[i] \}_{i = 1}^m$, the \emph{fixed-pole least squares realization problem} seeks to determine the model parameters $\vc{b} \in \Rset^{q + 1}$ of the unknown factor of the $n$th-order model that minimizes the $l_2$ norm of the misfit $\misfit = \y - \yhat$.
        It corresponds to the following optimization problem:
        \begin{equation}
            \label{eq:fixedpolesoptimizationproblem}
            \begin{gathered}
                \min_{\vc{b}, \yhat} \frac{1}{2} \norm{\misfit}^2_2 = \min_{\vc{b}, \yhat} \frac{1}{2} \norm{\y - \yhat}^2_2, \\
                \sub \toep{N - n}{\vc{c}} \toep{N - q}{\vc{b}} \yhat = \vc{0}, \left\langle \vc{b} , \vc{e} \right\rangle = 1, \\ 
            \end{gathered}
        \end{equation} 
        where the last constraint with adapted $\vc{e} \in \Rset^{q + 1}$ avoids the trivial solution $\vc{b} = \vc{0}$. 
    \end{definition}

    One could wonder whether an adaptation of the least squares methodology is truly necessary to be able to cope with fixed poles. 
    Would it not suffice to remove the effects of the given poles from the observed output data, i.e., to prefilter the data and use standard least squares realization techniques to estimate the remaining poles?
    We motivate our methodology by considering two (heuristic) prefiltering approaches. 

    \begin{example}[Motivational example]
        \label{ex:motivational}
        Given the output data
        \begin{equation}
            \label{eq:data}
            \y = 
            \begin{bmatrix}
                \num{3.0000} \\
                \num{5.0000} \\
                \num{2.0000} \\
                \num{3.0000} \\
                \num{4.0000} \\
                \num{2.0000} \\
                \num{3.0000} 
            \end{bmatrix},
        \end{equation}
        we want to estimate a second-order ($n = 2$) model, for which one of its poles is fixed by
        \begin{equation}
            \label{eq:fixedpole}
            \pole[1] = \num{-0.9557} 
        \end{equation}
        via one of the above-mentioned heuristics.\footnote{We adopt a four-digit precision convention for all floating-point numbers.}
        The misfit $\norm{\misfit}_2^2$ in terms of the coefficients $a_1$ and $a_2$ is visualized in \cref{fig:contour}.
        \begin{itemize}
            \item By using an orthogonal projection~\cite[Eq. (25)]{Lagauw2025}, we can compute the misfit $\misfit'$ associated with the first-order model defined by $\pole[1]$.
            We estimate the second pole by performing another $n = 1$ estimation, using the misfit $\misfit'$ as the `observed' output data.
            The most dominant pole present in the first-order misfit $\misfit'$, $\pole[2] = \num{0.8630}$, results in a misfit $\misfit$ with $\norm{\misfit}_2^2 = \num{8.4181}$.
            \item Alternatively, we can prefilter the observed output data $\y$ via the TSD technique described in~\cite{Majda1989}:
            \begin{equation}
                \y' = \toep{N-m}{\vc{c}} \y.
            \end{equation}
            After solving the standard least squares realization problem~\eqref{eq:fulloptimizationproblem} for a first-order model with $\y'$ as the `observed' output data, we find a second $\pole[2] = \num{0.9361}$, yielding a misfit $\norm{\misfit}_2^2 = \num{6.0070}$.
        \end{itemize}
        From \cref{fig:contour}, it is clear that these two heuristics do not find the global minimum and result in a larger misfit than the global minimum.
        The observation that both prefiltering approaches are suboptimal w.r.t.~\eqref{eq:fixedpolesoptimizationproblem} remains valid for other problem setups (i.e., different $N$, $n$, or $m$).
        This suboptimality is explained in \cref{rem:prefiltering}.
    \end{example}

\section{Methodology}
    \label{sec:methodology}

    Below, we characterize all minimizers of the fixed-pole optimization problem~\eqref{eq:fixedpolesoptimizationproblem} by adapting the standard globally optimal realization methodology from~\cite{Lagauw2025} and establishing a globally optimal realization approach for the fixed-pole setting.
    We start by reformulating the first-order necessary conditions for optimality as a matrix rank problem (\cref{sec:characterization}). 
    After showing that the minimizers result in a given rank (\cref{sec:lowerordersolutions}), we rephrase the fixed-pole problem as an MEP (\cref{sec:reformulation}).
    Finally, we discuss how the obtained MEP can be solved to find the globally optimal model subject to the fixed poles (\cref{sec:algorithms}).

    \subsection{Characterizing first-order optimality}
        \label{sec:characterization}

        With the introduction of Lagrange multipliers $\vc{\lambda} \in \Rset^{N - n}$ and $\mu \in \Rset$, the Lagrangian of~\eqref{eq:fixedpolesoptimizationproblem} is equal to
        \begin{equation}
            \begin{aligned}
                \mathcal{L}(\vc{b}, \yhat, \vc{\lambda}, \mu) &= \frac{1}{2} \norm{\misfit}_2^2 + \vc{\lambda}^\tr \toep{N - n}{\vc{c}} \toep{N - q}{\vc{b}} \yhat \\
                &+ \mu \left(\vc{b}^\tr \vc{e} - 1\right).
            \end{aligned}
        \end{equation} 
        The first-order necessary conditions for optimality can be obtained via the partial derivatives as
        \begin{align}
            \label{eq:fonc1}
             &\partial \mathcal{L} / \partial \vc{b} = \mthat{Y}^\tr (\toep{N - n}{\vc{c}})^\tr \vc{\lambda} + \mu \vc{e} = \vc{0}, \\
            \label{eq:fonc2}
             &\partial \mathcal{L} / \partial \yhat  = \yhat - \y + (\toep{N - q}{\vc{b}})^\tr (\toep{N - n}{\vc{c}})^\tr \vc{\lambda} = \vc{0}, \\
            \label{eq:fonc3}
             &\partial \mathcal{L} / \partial \vc{\lambda} = \toep{N - n}{\vc{c}} \toep{N - q}{\vc{b}} \yhat = \toep{N - n}{\vc{c}} \mthat{Y} \vc{b} = \vc{0}, \\ 
            \label{eq:fonc4}
             &\partial \mathcal{L} / \partial \mu = \vc{b}^\tr \vc{e} - 1 = 0,
        \end{align} 
        where $\mthat{Y} \in \Rset^{(N - q) \times (q + 1)}$ denotes the Hankel matrix constructed from the model-compliant output data $\yhat$.
        The real-valued solutions that satisfy~\eqref{eq:fonc1}--\eqref{eq:fonc4} are the \emph{critical points} of the problem.
        Multiplying~\eqref{eq:fonc1} from the left by $\vc{b}^\tr$ yields
        \begin{equation}
            \vc{b}^\tr \mthat{Y}^\tr (\toep{N - n}{\vc{c}})^\tr \vc{\lambda} + \mu \vc{b}^\tr \vc{e} = \vc{0} \iff \mu = 0,
        \end{equation} 
        by using that $\toep{N - n}{\vc{c}} \mthat{Y} \vc{b} = \vc{0}$, from~\eqref{eq:fonc3}, and $\vc{b}^\tr \vc{e} = 1$, from~\eqref{eq:fonc4}. 
        Therefore, we can set $\mu = 0$ from this point on. 
        As such,~\eqref{eq:fonc1} implies that
        \begin{equation}
            \label{eq:walsh}
            \mthat{Y}^\tr (\toep{N-n}{\vc{c}})^\tr \vc{\lambda} = \vc{0} \iff \vc{\lambda}^\tr \toep{N - n}{\vc{c}} \mthat{Y} = \vc{0}, 
        \end{equation} 
        meaning that the vector of Lagrange multipliers $\vc{\lambda}$ lies in the left kernel space of the filtered Hankel matrix $\mthat{Y}' = \toep{N - n}{\vc{c}} \mthat{Y} \in \Rset^{(N - n) \times (q + 1)}$. 
        We know from~\eqref{eq:fonc3} that the filtered Hankel matrix has to be rank-deficient, such that
        \begin{equation}
            \label{eq:hankelrank}
            \rank \mthat{Y}' \leq q.
        \end{equation} 
        Furthermore,~\eqref{eq:fonc3} shows that the output data $\yhat$ are model-compliant w.r.t. $a(z)$, while the filtered output data $\yhat' = \toep{N - n}{\vc{c}} \yhat$ are model-compliant w.r.t. $b(z)$.

    \subsection{Solutions with lower-order dynamics}
        \label{sec:lowerordersolutions}

        Similar as in the standard case without fixed poles~\cite{Lagauw2025}, we show that for all (local) minimizers of the fixed-pole optimization problem~\eqref{eq:fixedpolesoptimizationproblem} equality holds in~\eqref{eq:hankelrank}; a solution with lower-order dynamics can not be a (local) minimizer of~\eqref{eq:fixedpolesoptimizationproblem}.
        Strict inequality in~\eqref{eq:hankelrank} implies that the model-compliant output data $\yhat$ do not employ all degrees of freedom allowed by $a(z)$.
        At least one of the poles of $a(z)$ can be removed without changing the optimal $\yhat$.

        \begin{lemma}
            \label{lem:lowerordersolutions}
            Given output data $\yhat$ that comply with a model $a(z) = \, b(z) \prod_{i  = 1}^m (z - \pole[i])$, of which the order is $n$ and the parameters $\vc{b}$ correspond to a (local) minimizer of~\eqref{eq:fixedpolesoptimizationproblem}. 
            The corresponding filtered model-compliant output data $\yhat'$ have exactly $q$th-order dynamics, i.e., $\rank \mthat{Y}' = q$. 
        \end{lemma}

        \begin{proof}
            A formal proof of \cref{lem:lowerordersolutions} is omitted due to space constraints, but the necessary steps follow closely the reasoning in~\cite[Section~5]{Lagauw2025}.
            In particular, the multiplication of $\yhat$ with $\toep{N - n}{\vc{c}}$ annihilates the first $m$ terms in~\eqref{eq:restrictedimage}, so that $\vchat{y}' = \toep{N - n}{\vc{c}} \vchat{y} = \sum_{i = m + 1}^n \vc{v}_i x_i$ is determined solely by the zeros of $b(z)$.
            If $\rank \mthat{Y}'< q$ with $\mthat{Y}' = \toep{N - n}{\vc{c}} \mthat{Y}$, then $b(z)$ contains at least one non-contributing pole, i.e., a pole that does not influence $\yhat'$, and it can be removed without altering $\yhat$ or $\yhat'$.
            The orthogonality induced by the least squares measure implies that a small perturbation of the coefficients of $b(z)$ (which only modifies the non-contributing pole) leads to a strict reduction in the $l_2$ norm of the misfit.
        \end{proof}

    \subsection{Multiparameter eigenvalue problem}
        \label{sec:reformulation}

        If equality holds in~\eqref{eq:hankelrank}, then the filtered Hankel matrix $\mthat{Y}'$ has a left kernel space of dimension $ N - 2n + m$. 
        From~\eqref{eq:fonc3} and~\eqref{eq:swapped}, it follows that
        \begin{equation}
            \toep{N -2n + m}{\vc{b}} \toep{N - n}{\vc{c}} \mthat{Y} = \vc{0},
        \end{equation} 
        which, via~\eqref{eq:walsh}, implies that for all critical points for which equality holds in~\eqref{eq:hankelrank}, we can express $\vc{\lambda}$ as
        \begin{equation}
            \label{eq:doublefir}
            \vc{\lambda} = (\toep{N - 2n + m}{\vc{b}})^\tr \vc{g} 
        \end{equation} 
        for some $\vc{g} \in \Rset^{N - 2n + m}$.
        Substitution of~\eqref{eq:doublefir} into~\eqref{eq:fonc2} and multiplication from the left by $\toep{N - n}{\vc{a}} = \toep{N - n}{\vc{c}} \toep{N - q}{\vc{b}}$ leads to a characterization of this subset of critical points:
        \begin{equation}
            \label{eq:doublefirsystem}
            -\toep{N - n}{\vc{a}} \y + \toep{N - n}{\vc{a}} (\toep{N - n}{\vc{a}})^\tr (\toep{N - 2n + m}{\vc{b}})^\tr \vc{g} = \vc{0}.
        \end{equation} 
        Together with the non-triviality constraint~\eqref{eq:fonc4},~\eqref{eq:doublefirsystem} is a system of $N - n + 1$ (at most quartic) polynomial equations in the $N - n + 1$ variables $\vc{b}$ and $\vc{g}$.

        \begin{theorem}
            \label{thm:fixedpoles}
            Consider the MEP $\mcl{A}(\vc{b}) \vc{z} = \vc{0}$ defined by the $(N - n) \times (N - 2n + m + 1)$ cubic matrix polynomial 
            \begin{equation}
                \label{eq:matrixpolynomial}
                \mcl{A}(\vc{b}) = 
                \begin{bmatrix}
                    \toep{N - n}{\vc{c} \cdot \vc{b}} \y & \toep{N - n}{\vc{c} \cdot \vc{b}} (\toep{N - n}{\vc{c} \cdot \vc{b}})^\tr (\toep{N - 2n + m}{\vc{b}})^\tr 
                \end{bmatrix}, 
            \end{equation} 
            where $\toep{N - n}{\vc{c} \cdot \vc{b}} = \toep{N - n}{\vc{c}} \toep{N - q}{\vc{b}}$ depends only on $\vc{b}$ since $\vc{c}$ is fixed. 
            The MEP is non-square for typical problem sizes.
            The eigenvalues are values of $\vc{b} \in \Rset^{q + 1}$ for which $\mcl{A}(\vc{b})$ drops rank so that there exists a non-zero vector $\vc{z} \in \Cset^{N - 2n + m + 1}$.
            Each real-valued, affine eigenvalue of~\eqref{eq:matrixpolynomial} corresponds to a critical point of~\eqref{eq:fixedpolesoptimizationproblem} and the (local) minimizers are a subset of these affine eigenvalues.
        \end{theorem}

        \begin{proof}
            The multiparameter formulation in \cref{thm:fixedpoles} follows from \cref{lem:lowerordersolutions} and the opening paragraph of \cref{sec:reformulation}.
            When $\vc{a}$ is real, the product of the three Toeplitz matrices in~\eqref{eq:matrixpolynomial} is of full column rank, such that the first entry of the eigenvector $\vc{z}$ can always be normalized.
        \end{proof}

        When $\vc{e} = \begin{bmatrix} 0 & \cdots & 0 & 1 \end{bmatrix}^\tr$ is used as the normalization vector, the non-triviality constraint~\eqref{eq:fonc4} can be eliminated from the system by substituting $b_0 = 1$ in~\eqref{eq:doublefirsystem}.
        A similar adaptation of the MEP in~\eqref{eq:matrixpolynomial} is also possible.

        \begin{remark}
            Our methodology developed in~\cite{Lagauw2025} and the methodology that follows from \cref{thm:fixedpoles} are both globally optimal, but w.r.t. a different optimization problem. 
            To avoid confusion, we call the approach from~\cite{Lagauw2025} the standard globally optimal realization (S-GOR) technique and the approach that follows from \cref{thm:fixedpoles} the fixed-pole globally optimal realization (FP-GOR) technique.
        \end{remark}

        \begin{remark}
            \label{rem:prefiltering}
            Recall that in the TSD technique~\cite{Majda1989} the observed output data $\y$ are prefiltered to remove the effects of the fixed poles and standard least squares realization techniques are used on the filtered $\y'$. 
            As became evident in \cref{ex:motivational}, this approach is heuristic since the prefilter is applied to the observed output data $\y$ and not to the model-compliant output data $\yhat$, such that also the misfit is prefiltered. 
            The FP-GOR technique essentially applies the same prefilter, but it does so on the to-be-estimated model-compliant output data $\yhat$.
            By changing the order of the Toeplitz matrices in~\eqref{eq:filter}, as shown in~\eqref{eq:swapped}, we get that $\toep{N - n}{\vc{b}} \toep{N - m}{\vc{c}} \yhat = \toep{N - n}{\vc{b}} \vchat{y}' = \vc{0}$, illustrating that the dynamics of the prefiltered model-compliant output data $\vchat{y}' = \toep{N - m}{\vc{c}} \yhat$ are completely determined by the unknown poles $\lbrace \pole[i] \rbrace_{i = m + 1}^n$.
        \end{remark}

    \subsection{How to solve the MEP in \cref{thm:fixedpoles}?}
        \label{sec:algorithms}

        In this paper, we refrain ourselves of diving too deep into possible solution algorithms for MEPs.\footnote{Note that the multivariate polynomial system in~\eqref{eq:doublefirsystem} could also be tackled directly by multivariate polynomial root-finding algorithms (such as homotopy continuation methods and resultant-based approaches).}
        Two software toolboxes exist that solve these problems, namely \textsc{MacaulayLab} and \textsc{MultiParEig}.
        The former uses block Macaulay matrices to construct a joint generalized eigenvalue problem that has the affine eigenvalues of the MEP as its solutions~\cite{Vermeersch2022}, while the latter uses operator determinants or randomized sketching to reach that same goal~\cite{Hochstenbach2024}. 
        More information about the solution algorithms can be found in~\cite{Vermeersch2025a}.
        
        While solving the MEP in \cref{thm:fixedpoles}, it becomes apparent that the more poles are known beforehand, the less variables are contained in the MEP, such that the computational complexity of determining the eigenvalues decreases for increasing $m$.
        There are two interesting limit cases:
        \begin{itemize}
            \item \Cref{thm:fixedpoles} results in the MEP from the S-GOR technique when $m = 0$.
            \item For $q = 1$ and $b_0 = 1$, the matrix polynomial in~\eqref{eq:matrixpolynomial} is a square, univariate matrix polynomial in $b_1$. 
        \end{itemize}

\section{Numerical Examples}
    \label{sec:numericalexamples}

    The normalization vector $\vc{e} = \begin{bmatrix} 0 & \ldots & 0 & 1 \end{bmatrix}^\tr$ is used in each of the following numerical examples, so that the non-triviality constraint~\eqref{eq:fonc4} and the variable $b_0$ can be eliminated in~\eqref{eq:doublefirsystem} or~\eqref{eq:matrixpolynomial}.

    \begin{example}[\cref{ex:motivational} continued]
        \label{ex:continuation}
        Pick again the observed output data $\y$ from \cref{ex:motivational}. 
        Suppose that we want to estimate the globally optimal second-order ($n = 2$) model for which one of the poles ($m = 1$) is fixed by~\eqref{eq:fixedpole} (the fixed $\pole[1]$ is actually the pole of the globally optimal first-order model).
        The resulting multivariate polynomial system~\eqref{eq:doublefirsystem} has $\num{5}$ equations in the variables $b_1$ and $\vc{g}$, corresponding to a cubic polynomial eigenvalue problem~\eqref{eq:matrixpolynomial} with $5 \times 5$ coefficient matrices. 
        Out of the $\num{13}$ affine eigenvalues, only one is real-valued: $b_1 = \num{-0.9538}$.
        The globally optimal solution has $\pole[2] = \num{0.9538}$ as its second pole and the misfit $\norm{\misfit}_2^2$ is equal to $\num{5.9112}$. 
        The associated model is
        \begin{equation}
            \begin{aligned}
                a(z) &= (z - \pole[1]) b(z) = (z + \num{0.9557})(z - \num{0.9538}) \\
                &= z^2 + \num{0.0019} z - \num{0.9116}.
            \end{aligned}
        \end{equation}
        
        Now, consider the standard least squares problem where both poles can be chosen freely, i.e., we do not fix the pole $\pole[1]$ beforehand.
        The globally optimal solution, $\pole[1] = \num{-0.5351}$ and  $\pole[2] = \num{0.9194}$, is found by solving a cubic two-parameter eigenvalue problem (via the S-GOR technique).
        The global minimum of this $n = 2$ fit, $\norm{\misfit}_2^2 = \num{3.8836}$, is smaller than the global minimum of the $n = 1$ fit with one given pole. 

        \Cref{tab:overview} contains the best model for each of the applied techniques.
        Although (global) optimality is achieved in each intermediate step (both for the $(n, m) = (1, 0)$ and $(n, m) = (2, 1)$ fit), we observe that the $n$th-order model obtained by concatenating the poles of FP-GOR differs from the globally optimal solution obtained via the S-GOR technique. 
    \end{example}
    
    \begin{table}
        \centering
        \caption{Numerical results when solving \cref{ex:continuation} via the different realization approaches. The naive prefilter (NPF) approach uses the optimal misfit associated with the first-order model, while the TSD prefiltering technique uses the heuristic from~\cite{Majda1989}. The FP-GOR technique finds the optimal solution subject to one fixed pole, but the S-GOR technique is not restricted by any \textit{a priori} information.} 
        \label{tab:overview}
        \begin{tabular}{ccccc}
            \toprule
            \textbf{approach} & $\bm{(n,m)}$ & $\bm{\norm{\misfit}_2^2}$ & $\bm{\pole[1]}$ & $\bm{\pole[2]}$ \\ 
            \midrule
            NPF (heuristic) & $(2,1)$ & $\tnum{8.4181}$ & $\tnum{-0.9557}^\P$ & $\tnum[2.4]{0.8630}$ \\
            TSD (heuristic) & $(2,1)$ & $\tnum{6.0070}$ & $\tnum{-0.9557}^\P$ & $\tnum[2.4]{0.9361}$ \\
            FP-GOR (optimal) & $(2,1)$ & $\tnum{5.9112}$ & $\tnum{-0.9557}^\P$ & $\tnum[2.4]{0.9538}$ \\
            \midrule
            S-GOR (optimal) & $(2,0)$ & $\tnum{3.8836}$ & $\tnum{-0.5351}\phantom{^\P}$ & $\tnum[2.4]{0.9194}$ \\ 
            \bottomrule
        \end{tabular} \\[3pt]
        \footnotesize \noindent $^\P$The pole $\pole[1] = \num{-0.9557}$ was fixed in the experiment.
    \end{table}

    \Cref{ex:continuation} illustrates a broader principle: the fact that each intermediate step in a methodology is executed optimally does not necessarily imply that the overall sequence of steps constitutes an optimal procedure.

    \begin{example}
        \label{ex:sanitycheck}
        We now perform an $(n, m) = (2, 1)$ estimate with one of the poles obtained form the S-GOR technique, i.e., $(n, m) = (2, 0)$.
        We take $\pole[1] = \num{-0.5351}$ from \cref{tab:overview} and use the FP-GOR technique to determine a globally optimal estimate for $\pole[2]$. 
        The result of solving~\eqref{eq:doublefirsystem} or~\eqref{eq:matrixpolynomial} is $\pole[2] = \num{0.9194}$, which coincides with the globally optimal result obtained via the S-GOR technique.
    \end{example}

    During the next example, we verify whether the motivation for the adapted methodology was correct: ``Does incorporating \textit{a priori} information about a subset of the poles lead to an improved estimation w.r.t.~the underlying exact system output dynamics in a noisy context?''

    \begin{example}
        \label{ex:statistics}

        \newcommand{\statexpmm}{\num{50}}
        \newcommand{\statexpmn}{\num{15}}
        \newcommand{\statexpms}{\ensuremath{\{0.05, 0.15, \allowbreak \ldots, 0.45\}}}
        \newcommand{\statexpml}{\ensuremath{-0.75}}

        Consider the third-order $n = 3$ autonomous dynamical model with poles $\pole[1,2] = e^{\pm 0.8j}$ and $\pole[3] = \statexpml$ defined by the state-space model $(\mt{A}, \mt{C})$:
        \begin{equation}
            \mt{A} = \mt{T}^\inv \!
            \begin{bmatrix}
                \operatorname{Re}(\pole[1]) & -\operatorname{Im}(\pole[1]) & 0 \\ 
                \operatorname{Im}(\pole[1]) & \phantom{-}\operatorname{Re}(\pole[1]) & 0 \\
                0 & \phantom{-}0 & \pole[3]
            \end{bmatrix} 
            \mt{T}, \quad \mt{C} = 
            \begin{bmatrix}
                2 & 2 & 2
            \end{bmatrix},
        \end{equation} 
        where $\mt{T} \in \Rset^{3 \times 3}$ can be any non-singular matrix.
        We can use this model to generate the exact output data $\ytrue \in \Rset^N$:
        \begin{equation}
            \label{eq:truedata}
            x_k = \mt{C} \mt{A}^{k} \vc{x_0}, \quad k = 0, 1, \dots, N - 1, 
        \end{equation}
        with initial state $\vc{x}_0 = \begin{bmatrix} 1 & 1 & 1 \end{bmatrix}^\tr$.
        The observed output data $\y \in \Rset^N$ are obtained from the following noisy setup:
        \begin{equation}
            \label{eq:givendata}
            y_k = x_k + \sigma \epsilon_k, \quad k = 0, 1, \dots, N - 1, 
        \end{equation} 
        where $\epsilon_k \sim \mathcal{N}(0,1)$ such that $\sigma \epsilon_k \sim \mathcal{N}(0,\sigma^2)$. 
        
        Given $\statexpmm$ different realizations of length $N = \statexpmn$ for every noise level $\sigma \in \statexpms$, we compute estimates $\estpole[3]$ and the corresponding model-compliant output data $\yhat^{(i)}$ via two approaches: firstly, by using the S-GOR technique to find all three poles without \textit{a priori} information, i.e., $(n, m) = (3, 0)$, and secondly, by using the FP-GOR technique, in which we assume the complex pole pair $\pole[1,2] = e^{\pm 0.8j}$ to be fixed while computing $\estpole[3]$, i.e., $(n, m) = (3, 2)$.

        \ifarxiv
            \begin{figure}
                \centering
                \begin{subfigure}[t]{\textwidth}
                    \centering
                    \pgfplotsset{legend image code/.code = {\draw[#1] (0cm,-0.1cm) rectangle (0.6cm,0.1cm);}}

\begin{tikzpicture}[trim axis left, trim axis right]
    \begin{axis}[
        boxplotstyle,
        xmin = 0, 
        xmax = 2,
        ymin = 0.5, 
        ymax = 5.5,
        ylabel = {noise level $\sigma$},
        xlabel = {$\Vert\y^{(i)} - \yhat^{(i)}\Vert_2^2$},
    ]

        \addplot+[boxplot prepared = {
            median = 0.1526,
            upper quartile = 0.1769,
            lower quartile = 0.1332,
            upper whisker = 0.2177,
            lower whisker = 0.0677,
            draw position = 0.85
        }] coordinates {};
        \label{plot:sgor}

        \addplot+[boxplot prepared = {
            median = 0.1722,
            upper quartile = 0.2028,
            lower quartile = 0.1509,
            upper whisker = 0.2503,
            lower whisker = 0.0865,
            draw position = 1.15
        }] coordinates {};
        \label{plot:fpgor}

        %
      
        \addplot+[boxplot prepared = {
            median = 0.4605,
            upper quartile = 0.5313,
            lower quartile = 0.4007,
            upper whisker = 0.6533,
            lower whisker = 0.2191,
            draw position = 1.85
        }] coordinates {};

        \addplot+[boxplot prepared = {
            median = 0.5111,
            upper quartile = 0.6046,
            lower quartile = 0.4543,
            upper whisker = 0.7509,
            lower whisker = 0.2757,
            draw position = 2.15
        }] coordinates {};

        %
      
        \addplot+[boxplot prepared = {
            median = 0.7688,
            upper quartile = 0.8835,
            lower quartile = 0.6693,
            upper whisker = 1.0891,
            lower whisker = 0.3781,
            draw position = 2.85
        }] coordinates {};

        \addplot+[boxplot prepared = {
            median = 0.8512,
            upper quartile = 1.0040,
            lower quartile = 0.7597,
            upper whisker = 1.2516,
            lower whisker = 0.4779,
            draw position = 3.15
        }] coordinates {};

        %
      
        \addplot+[boxplot prepared = {
            median = 1.0779,
            upper quartile = 1.2176,
            lower quartile = 0.9369,
            upper whisker = 1.5250,
            lower whisker = 0.5426,
            draw position = 3.85
        }] coordinates {};

        \addplot+[boxplot prepared = {
            median = 1.1889,
            upper quartile = 1.3700,
            lower quartile = 1.0668,
            upper whisker = 1.5765,
            lower whisker = 0.6870,
            draw position = 4.15
        }] coordinates {};

        %
      
        \addplot+[boxplot prepared = {
            median = 1.3365,
            upper quartile = 1.5349,
            lower quartile = 1.2028,
            upper whisker = 1.9611,
            lower whisker = 0.7106,
            draw position = 4.85
        }] coordinates {};

        \addplot+[boxplot prepared = {
            median = 1.5088,
            upper quartile = 1.7163,
            lower quartile = 1.3752,
            upper whisker = 2.0123,
            lower whisker = 0.9002,
            draw position = 5.15
        }] coordinates {};

        \addplot+[only marks, mark = *, mark size = 1pt, draw = black, fill = black] coordinates {(0.0610, 0.85)}; 
        \addplot+[only marks, mark = *, mark size = 1pt, draw = black, fill = black] coordinates {(0.2011, 1.85)}; 
        \addplot+[only marks, mark = *, mark size = 1pt, draw = black, fill = black] coordinates {(0.1850, 1.85)}; 
        \addplot+[only marks, mark = *, mark size = 1pt, draw = black, fill = black] coordinates {(0.3324, 2.85)}; 
        \addplot+[only marks, mark = *, mark size = 1pt, draw = black, fill = black] coordinates {(0.3112, 2.85)}; 
        \addplot+[only marks, mark = *, mark size = 1pt, draw = black, fill = black] coordinates {(0.4622, 3.85)}; 
        \addplot+[only marks, mark = *, mark size = 1pt, draw = black, fill = black] coordinates {(0.4391, 3.85)}; 
        \addplot+[only marks, mark = *, mark size = 1pt, draw = black, fill = black] coordinates {(0.5911, 4.85)}; 
        \addplot+[only marks, mark = *, mark size = 1pt, draw = black, fill = black] coordinates {(0.5683, 4.85)}; 
    \end{axis}
\end{tikzpicture}
                    \caption{comparison to given data}
                    \label{fig:statistics:costvalues}
                \end{subfigure}
                \begin{subfigure}[t]{\textwidth}
                    \centering
                    \begin{tikzpicture}[trim axis left, trim axis right]
    \begin{axis}[
        boxplotstyle,
        xmin = 0, 
        xmax = 2,
        ymin = 0.5, 
        ymax = 5.5,
        ylabel = {noise level $\sigma$},
        xlabel = {$\Vert\ytrue - \yhat^{(i)}\Vert_2^2$},
    ]
    
        \addplot+[boxplot prepared = {
            median = 0.1195,
            upper quartile = 0.1428,
            lower quartile = 0.1020,
            upper whisker = 0.2016,
            lower whisker = 0.0689,
            draw position = 0.85
        }] coordinates {};

        \addplot+[boxplot prepared = {
            median = 0.0908,
            upper quartile = 0.1135,
            lower quartile = 0.0807,
            upper whisker = 0.1553,
            lower whisker = 0.0466,
            draw position = 1.15
        }] coordinates {};

        %
      
        \addplot+[boxplot prepared = {
            median = 0.3627,
            upper quartile = 0.4292,
            lower quartile = 0.3076,
            upper whisker = 0.5344,
            lower whisker = 0.2057,
            draw position = 1.85
        }] coordinates {};

        \addplot+[boxplot prepared = {
            median = 0.2685,
            upper quartile = 0.3347,
            lower quartile = 0.2437,
            upper whisker = 0.4492,
            lower whisker = 0.1395,
            draw position = 2.15
        }] coordinates {};

        %
      
        \addplot+[boxplot prepared = {
            median = 0.6062,
            upper quartile = 0.7119,
            lower quartile = 0.5304,
            upper whisker = 0.9025,
            lower whisker = 0.3414,
            draw position = 2.85
        }] coordinates {};

        \addplot+[boxplot prepared = {
            median = 0.4438,
            upper quartile = 0.5640,
            lower quartile = 0.4071,
            upper whisker = 0.7227,
            lower whisker = 0.2323,
            draw position = 3.15
        }] coordinates {};

        %
      
        \addplot+[boxplot prepared = {
            median = 0.8726,
            upper quartile = 1.0386,
            lower quartile = 0.7445,
            upper whisker = 1.4656,
            lower whisker = 0.4763,
            draw position = 3.85
        }] coordinates {};

        \addplot+[boxplot prepared = {
            median = 0.6464,
            upper quartile = 0.8755,
            lower quartile = 0.5846,
            upper whisker = 1.1925,
            lower whisker = 0.3252,
            draw position = 4.15
        }] coordinates {};

        %
      
        \addplot+[boxplot prepared = {
            median = 1.2098,
            upper quartile = 1.4833,
            lower quartile = 0.9829,
            upper whisker = 2.1241,
            lower whisker = 0.6104,
            draw position = 4.85
        }] coordinates {};

        \addplot+[boxplot prepared = {
            median = 0.8533,
            upper quartile = 1.2248,
            lower quartile = 0.7517,
            upper whisker = 1.7816,
            lower whisker = 0.4179,
            draw position = 5.15
        }] coordinates {};

        \addplot+[only marks, mark = *, mark size = 1pt, draw = black, fill = black] coordinates {(0.3275, 1.15)}; 
        
        \addplot+[only marks, mark = *, mark size = 1pt, draw = black, fill = black] coordinates {(0.3636, 1.15)}; 

        \addplot+[only marks, mark = *, mark size = 1pt, draw = black, fill = black] coordinates {(0.6181, 1.85)}; 

        \addplot+[only marks, mark = *, mark size = 1pt, draw = black, fill = black] coordinates {(0.4913, 2.15)}; 

        \addplot+[only marks, mark = *, mark size = 1pt, draw = black, fill = black] coordinates {(0.5675, 2.15)}; 
        
        \addplot+[only marks, mark = *, mark size = 1pt, draw = black, fill = black] coordinates {(0.4833, 2.15)}; 


        


        \addplot+[only marks, mark = *, mark size = 1pt, draw = black, fill = black] coordinates {(0.1638, 1.15)}; 

        \addplot+[only marks, mark = *, mark size = 1pt, draw = black, fill = black] coordinates {(0.1729, 1.15)}; 

        \addplot+[only marks, mark = *, mark size = 1pt, draw = black, fill = black] coordinates {(2.0988, 5.15)}; 

        \addplot+[only marks, mark = *, mark size = 1pt, draw = black, fill = black] coordinates {(1.0432, 2.85)}; 

        \addplot+[only marks, mark = *, mark size = 1pt, draw = black, fill = black] coordinates {(0.8186, 3.15)}; 

        \addplot+[only marks, mark = *, mark size = 1pt, draw = black, fill = black] coordinates {(0.9849, 3.15)}; 
     
        \addplot+[only marks, mark = *, mark size = 1pt, draw = black, fill = black] coordinates {(0.8391, 3.15)}; 






        \addplot+[only marks, mark = *, mark size = 1pt, draw = black, fill = black] coordinates {(1.5440, 3.85)}; 

        \addplot+[only marks, mark = *, mark size = 1pt, draw = black, fill = black] coordinates {(1.7795, 3.85)}; 

        \addplot+[only marks, mark = *, mark size = 1pt, draw = black, fill = black] coordinates {(1.5090, 3.85)}; 

        \addplot+[only marks, mark = *, mark size = 1pt, draw = black, fill = black] coordinates {(1.5120, 4.15)}; 

        \addplot+[only marks, mark = *, mark size = 1pt, draw = black, fill = black] coordinates {(1.7688, 4.15)}; 

        \addplot+[only marks, mark = *, mark size = 1pt, draw = black, fill = black] coordinates {(1.3913, 4.15)}; 
        
        \addplot+[only marks, mark = *, mark size = 1pt, draw = black, fill = black] coordinates {(1.4360, 4.15)}; 
        

        
        
    \end{axis}
\end{tikzpicture}
                    \caption{comparison to exact data}
                    \label{fig:statistics:exactvalues}
                \end{subfigure}
                \caption{Analysis of the model-compliant output data $\yhat^{(i)}$ for the different realization experiments, $i = 1, 2, \ldots, 50$, in \cref{ex:statistics}. One pole of a third-order model is estimated by the standard (S-GOR) and fixed-pole (FP-GOR) realization approach from given output data $\y \in \Rset^{\statexpmn}$, which is generated using~\eqref{eq:givendata} for noise levels $\sigma \in \statexpms$. From \cref{fig:statistics:costvalues} it seems that the S-GOR technique (\ref{plot:sgor}) results in a smaller misfit $\Vert\y^{(i)} - \yhat^{(i)}\Vert_2^2$ than the FP-GOR technique (\ref{plot:fpgor}). However, the latter obtains better results when comparing $\yhat^{(i)}$ with the exact data $\ytrue$ generated by the underlying model, as shown in \cref{fig:statistics:exactvalues}.}
                \label{fig:statistics}
            \end{figure}

            \begin{figure}
                \centering
                \begin{tikzpicture}[trim axis left, trim axis right]
    \begin{axis}[
        figurestyle, 
        xlabel={realization $i$ \vphantom{$\norm{\y - \yhat}_2^2$}},
        ylabel= {realization $j$},
        view={0}{90},
        xmin = 1,
        xmax = 50,
        ymin = 1,
        ymax = 50,
        xtick = {1,10,20,30,40,50},
        ytick = {1,10,20,30,40,50},
        colormap = {gradient}{color = (black) color = (blue) color = (white) color = (red) color = (black)},
        colorbar,
        colorbar style = {
            colorbarstyle,
            title = {$d$}, 
            ymin = -13, 
            ymax = 13, 
            ytick = {-13, 0, 13},
            yticklabels = {$<0$, $=0$, $>0$}
        },
    ]
        \addplot3 [surf, mesh/cols = 50, mesh/rows = 50] table[col sep = comma] {figures/crossdata.dat};
    \end{axis}
\end{tikzpicture}
                \caption{Cross-validation for noise level $\sigma = 0.25$. It visualizes the difference between S-GOR and FP-GOR, $d = \Vert\misfit^{(i,j)}_\text{S-GOR}\Vert_2^2 - \Vert\misfit^{(i,j)}_\text{FP-GOR}\Vert_2^2$, of the misfit $\Vert\misfit^{(i,j)}\Vert_2^2 = \Vert\y^{(i)} - \yhat^{(j)}\Vert_2^2$ for all values $i,j = 1, 2, \ldots, 50$. This qualitative measure visualizes when the model obtained by FP-GOR describes the given data better ($d > 0$) or worse ($d < 0$) than S-GOR for other realizations ($i \neq j$). On the off-diagonal, the FP-GOR misfit is generally smaller than the S-GOR misfit, indicating a better generalization.}
                \label{fig:statistics:crossvalidation}
            \end{figure}
        \fi

        For the S-GOR technique, the MEP has $\num{6466}$ solutions, out of which only a small subset is real-valued, while the FP-GOR technique requires solving a univariate problem that has only $\num{37}$ solutions. 
        The standard approach is computationally much more expensive than the fixed-pole approach. 

        It is clear from \cref{fig:statistics:costvalues} that the misfits $\Vert \y^{(i)} - \yhat^{(i)}\Vert_2^2$, $i = 1, 2, \ldots, \statexpmm$, of the S-GOR approach are consistently lower than the ones obtained by using the FP-GOR approach.
        This is as expected since the former is guaranteed to reach the global minimum of the standard least squares realization problem~\eqref{eq:fulloptimizationproblem}, which is a lower-bound for the fixed-pole least squares realization problem~\eqref{eq:fixedpolesoptimizationproblem}.
        Indeed, fixing the poles $\pole[1,2]$ reduces the degrees of freedom.
        For this specific problem setup, the estimated $\estpole[3]$ appear typically closer to the actual $\pole[3] = -0.75$ with S-GOR than with FP-GOR.
        This occurs because S-GOR can distribute noise influence across all three poles, while FP-GOR has to concentrate it only in $\estpole[3]$.

        However, in practice, minimizing the difference to the exact output data, $\Vert\ytrue - \yhat^{(i)}\Vert_2^2$, is more valuable than the misfit, $\Vert\y^{(i)} - \yhat^{(i)}\Vert_2^2$, though the former can only be verified when the exact output data is available.
        \cref{fig:statistics:exactvalues} indicates that the model-compliant output data $\yhat^{(i)}$ obtained by using the FP-GOR technique are significantly closer to the noise-free output data $\ytrue$ than with the S-GOR technique. 
        Incorporating \textit{a priori} information about the pole pair $\pole[1,2]$ enhances the model's ability to distinguish signal from noise, thereby improving estimation accuracy.

        Moreover, the cross-validation analysis in \cref{fig:statistics:crossvalidation} demonstrates that S-GOR overfits to the noise of individual output data sequences, while the \textit{a priori} information in FP-GOR generalizes better across the different realizations; its model-compliant output data $\yhat^{(i)}$ align better with the other $\y^{(j)}$ output data sequences ($i \neq j$).
    \end{example}

\section{Conclusions and Future Work}
    \label{sec:conclusion}

    We presented an adaptation of the standard least squares realization problem that exploits fixed-pole information in the identification of autonomous, single-output, linear time-invariant dynamical models.
    By embedding \textit{a priori} information into the globally optimal formulation of~\cite{Lagauw2025}, we adapted the (rectangular) multiparameter formulation to the fixed-pole setting. 
    This formulation allows computing all local and global minimizers of the constrained problem, establishing a globally optimal methodology for the identification of the unknown poles and outperforming existing heuristics.
    Moreover, as some of the poles of the model are fixed, the number of unknowns decreases, which in turn leads to a reduction in computational complexity compared to~\cite{Lagauw2025} for obtaining the globally optimal solution(s).
    We validated the proposed approach through a set of numerical examples, one of which illustrated that the model-compliant output sequences obtained by incorporating \textit{a priori} information more accurately approximate the exact, noise-free signal than those produced by the standard formulation.

    This investigation gives rise to several opportunities for future work. 
    Firstly, our experimental analysis was limited in scope. 
    A more extensive evaluation of the estimator under a broader range of conditions will be of significant interest. 
    In particular, we expect that for sufficiently large differences $q = n - m$, the fixed-pole approach will lead to more accurate estimates of the most dominant unknown poles.
    It would also be interesting to investigate how sensitive these estimates are w.r.t.~perturbations in the fixed-pole information.
    Secondly, the computational properties of the multiparameter formulations need further investigation, particularly in terms of scalability and  stability for large, realistic problems.

    \bibliographystyle{plain}
    \bibliography{references}

\begin{thebibliography}{10}

\bibitem{Auton1981}
Jon~R. Auton and Michael~L. Van~Blaricum.
\newblock Investigation of procedures for automatic resonance extraction from
  noisy transient electromagnetics data.
\newblock Technical report, Effects Tech. Inc., Santa Barbara, CA, USA, 1981.

\bibitem{Chen1997a}
Hua Chen, Sabine Van~Huffel, Ad~van~den Boom, and Paul van~den Bosch.
\newblock Subspace-based parameter estimation of exponentially damped sinusoids
  using prior knowledge of frequency and phase.
\newblock {\em Signal Processing}, 59(1):129--136, 1997.

\bibitem{Chen1997}
Hua Chen, Sabine Van~Huffel, and Joos Vandewalle.
\newblock Improved methods for exponential parameter estimation in the presence
  of known poles and noise.
\newblock {\em IEEE Transactions on Signal Processing}, 45(5):1390--1393, 1997.

\bibitem{DeMoor2020}
Bart De~Moor.
\newblock Least squares optimal realisation of autonomous {LTI} systems is an
  eigenvalue problem.
\newblock {\em Communications in Information and Systems}, 20(2):163--207,
  2020.

\bibitem{Dowling1994}
Eric~M. Dowling, Ronald~D. DeGroat, and Darel~A. Linebarger.
\newblock Exponential parameter estimation in the presence of known components
  and noise.
\newblock {\em IEEE Transactions on Antennas and Propagation}, 42(5):590--599,
  1994.

\bibitem{Hauer1990}
John~F. Hauer, Cedric~J. Demeure, and Louis~L. Scharf.
\newblock Initial results in {Prony} analysis of power system response signals.
\newblock {\em IEEE Transactions on Power Systems}, 5(1):80--89, 1990.

\bibitem{Hochstenbach2024}
Michiel~E. Hochstenbach, Toma\v{z} Ko\v{s}ir, and Bor Plestenjak.
\newblock Numerical methods for rectangular multiparameter eigenvalue problems,
  with applications to finding optimal {ARMA} and {LTI} models.
\newblock {\em Numerical Linear Algebra with Applications}, 31(2):e2540:1--23,
  2024.

\bibitem{Hua1991}
Yingbo Hua and Tapan~K. Sarkar.
\newblock On {SVD} for estimating generalized eigenvalues of singular matrix
  pencil in noise.
\newblock {\em IEEE Transactions on Signal Processing}, 39(4):892--900, 1991.

\bibitem{Lagauw2025}
Sibren Lagauw, Lukas Vanpoucke, and Bart De~Moor.
\newblock Exact solution to the least squares realization problem as a
  multiparameter eigenvalue problem.
\newblock {\em Automatica}, 185:112792:1--11, 2026.

\bibitem{Majda1989}
George Majda and Musheng Wei.
\newblock A simple procedure to eliminate known poles from a time series.
\newblock {\em IEEE Transactions on Antennas and Propagation},
  37(10):1343--1344, 1989.

\bibitem{Vermeersch2022}
Christof Vermeersch and Bart De~Moor.
\newblock Two complementary block {M}acaulay matrix algorithms to solve
  multiparameter eigenvalue problems.
\newblock {\em Linear Algebra and its Applications}, 654:177--209, 2022.

\bibitem{Vermeersch2025a}
Christof Vermeersch and Bart De~Moor.
\newblock On the rectangular multiparameter eigenvalue problem.
\newblock Technical report, KU Leuven, Leuven, Belgium, 2025.
\newblock \url{https://ftp.esat.kuleuven.be/pub/stadius/cvermeer/25-89.pdf}.

\end{thebibliography}
\end{document}